\documentclass{amsart}
\usepackage{amssymb,amsmath,amsthm,amsfonts}
\usepackage{framed,comment,enumerate}
\usepackage[pdftex]{graphicx}
\usepackage{epstopdf}
\usepackage{xcolor, framed}
\usepackage{color}
\def\R {\mathbb{R}}
\def\eps{\varepsilon}

\def\MA {Monge-Amp\`ere }

\def\Shcpsi {S_h^c[\psi]}

\def\Shcu {S_h^c[u]}
\def\Shu{S_h[u]}
\def\Shcv {S_h^c[v]}

\def\LTO{L_{\Omega_1}}
\def\RTO{R_{\Omega_1}}
\def\LTT{L_{\Omega_2}}
\def\RTT{R_{\Omega_2}}

\def\DDX{\frac{\partial}{\partial x_1}}
\def\DDY{\frac{\partial}{\partial x_2}}

\def\EL{e_{long}}

\newtheorem{prop}{Proposition}[section]
\newtheorem{thm}{Theorem}[section]

\newtheorem{lem}{Lemma}[section]
\theoremstyle{definition}
\newtheorem{defi}{Definition}[section]
\newtheorem{rem}{Remark}[section]
\numberwithin{equation}{section}

\title{Regularity of optimal transport between planar convex domains}

\author{Ovidiu Savin}
\address{Department of Mathematics,	Columbia University, New York, USA}
\email{savin@math.columbia.edu}

\author{Hui Yu}
\address{Department of Mathematics,	Columbia University, New York, USA}
\email{ huiyu@math.columbia.edu}
\thanks{O.~S.~is supported by  NSF grant DMS-1500438.}

\begin{document}

\begin{abstract}
For $0<p<+\infty$, we prove a global $W^{2,p}$-estimate for potentials of optimal transport maps between convex domains in the plane.  Among the tools developed for that purpose are obliqueness  in general convex domains and estimates for the growth of eccentricity of sections of the potentials. 
\end{abstract}

\maketitle

\section{Introduction}
Given domains $U_1$ and $U_2$ in $\R^d$ with the same volume, the optimal transport, in its most basic form, is a map $T:U_1\to U_2$ that minimizes the cost of transportation $$\int_{U_1}|Tx-x|^2dx$$ over all measure preserving maps $T$ from $U_1$ to $U_2$.   Besides its intrinsic interest, the past two decades witnessed an almost explosive amount of applications of this theory to probability, geometry, PDEs, and many other branches of mathematics which  a priori do not seem related, see for example the book by Villani \cite{V}.  

Part of this popularity is due to  the pioneering work by Brenier \cite{B}, which contains a very flexible existence theory as well as the fact that $T=\nabla \psi$ for some convex function $\psi$, which is often called the potential of $T$. The regularity of this map turns out to be much more delicate. Though partial regularity can be established for general domains (see \cite{DPF},\cite{FK},\cite{GO}), it was observed by Caffarelli \cite{CMap1} that even for the continuity of $T$, convexity of $U_1$ and $U_2$ becomes necessary.  Under this convexity assumption  it was proved in the same work that the map is smooth in the \textit{interior} of $U_1$, following from the key observation that the potential $\psi$ solves the \MA  equation \begin{equation}\det(D^2 \psi)=\chi_{U_1} \text{ in $\R^d$}\end{equation} in the Alexandrov sense, which makes the theory developed by Caffarelli in \cite{CMA1}, \cite{CMA2}, \cite{CMA3} applicable.  Convexity of the domains and (1.1) also imply the doubling property of the \MA measure of $\psi$ at $\partial U_1$. Exploiting this,  Caffarelli proved in \cite{CMap2} that $T$ is  H\"older continuous \textit{up to the boundary} of $U_1$ for some small H\"older exponent $\delta>0$. 

For general dimensions $d$, to go beyond this global $C^{\delta}$-estimate seems to require more regularity of the domains, since at points near $\partial U_1$ the sections/ level sets of $\psi$ are heavily influenced by the geometry of the boundary. For $C^2$ and uniformly convex domains, Caffarelli showed in \cite{CMap3} that $T \in C^{1,\alpha}$ up to the boundary. Independently, Urbas \cite{U} obtained the same result under the slightly stronger $C^3$ condition on the domains. Very recently the regularity assumptions on the two domains were weakened by  Chen-Liu-Wang \cite{CLW}  to $C^{1,1}$ and convexity. 

The purpose of this work is to show that in the plane one can go beyond $C^{\delta}$ without  further assumption on the domains other than convexity. To be precise, our main result is the following: 
\begin{thm}
Let $U_1$ and $U_2$ be bounded convex domains in $\R^2$ of area 1, and let $\psi$ be the potential for the optimal transport between $U_1$ and $U_2$. 

Given $\eps>0$, we have
$$\|D^2\psi(x)\| \le C(\eps) \, \, dist(x,\partial U_1)^{-\eps}, \quad \quad \forall x \in U_1,$$ for some constant $C(\eps)$ depending on $\eps$ and the maximal diameters of $U_1$ and $U_2$. 

In particular, given any $p<\infty$, $$\|D^2\psi\|_{\mathcal{L}^p(\overline{U_1})}\le C(p)$$ for some $C(p)$ depending only on $p$ and the diameters of $U_1$ and $U_2.$
\end{thm} 
This gives global $C^{\alpha}$-regularity of the optimal transport for any $\alpha\in(0,1)$.

To see why such a global estimate can be quite subtle, one might draw a comparison with the  Dirichlet problem as in Wang \cite{Wang}, Trudinger-Wang \cite{TW}, and more recent works of the first-named author \cite{Savin1}, \cite{Savin2}. In all these works, strong regularity of the boundary ($C^2$ or $C^3$) are needed to tame the influence of the boundary on the geometry of sections. To get estimates in very rough domains as in our case requires new ideas, and these ideas and tools developed here will hopefully prove valuable for future study of \MA equation in domains with low regularity. 

Unlike the Dirichlet problem, the natural boundary condition for our problem is the so-called second boundary condition, namely, \begin{equation}
\nabla \psi (\partial U_1)=\partial U_2.
\end{equation}  It was observed by Caffarelli \cite{CMap3} and Urbas \cite{U} that in smooth domains this condition implies obliqueness, that is, the angle between the normal at $x_0\in\partial U_1$ and the normal at $\nabla \psi(x_0)\in\partial U_2$ is uniformly bounded away from $\pi/2$. Thus up to an affine transformation, $\partial U_1$ and $\partial U_2$ cut the sections in the same direction at corresponding points. 

To get obliqueness for general convex domains, one first needs a replacement for normal vectors at non-differentiable points at the boundary. Our choice is the left and right tangent rays, which are respectively the critical supporting rays to the domain in the clockwise and counter-clockwise direction. The precise definition is given in the third section.   In the same section, it is shown that at corresponding points, the angle between these tangent rays are bounded away from $\pi/2$.

Together with the duality between the sections of $\psi$ and of the potential of the inverse of $T$, this obliqueness leads to a growth control over the eccentricity of the sections of $\psi$, which is equivalent to a pointwise $C^{1,\alpha}$ estimate, for any $\alpha \in (0,1)$. The main result follows by compactness by applying such an argument to a family of \textit{normalized solutions}, depending only on the inner and outer radii of the domains.

This paper is organized as follows. In the second section the reader can find some preliminary results and definitions that are used throughout the paper. In particular we introduce a compact family of solutions $\mathcal S(\bar \delta)$ that contain our potentials as well as their renormalizations. Most of the estimates are written in terms of the geometry of the sections of the potentials belonging to this family.  In the third section is the proof for obliqueness, which is used in the fourth section to control the growth of the eccentricity. In the last section we combine all these ingredients and give the proof of the main result.

\section{Sections, ellipses and the family of normalized solutions}
To simplify certain statements, we first introduce some geometric notions.

Given an ellipse $E$, we write $E=x_0+\{\lambda e_{short}+\Lambda e_{long}\},$ where $x_0$ is the centre,  $\lambda$ and $\Lambda$ denote the lengths of the long and short axises, and $e_{short}$ and $e_{long}$ are the directions of the corresponding axises. $E^\perp$ denotes the perpendicular ellipse, namely, $E^\perp=x_0+\{\Lambda e_{short}+\lambda e_{long}\}.$
  
The shape of an ellipse is described by the following quantity: 
\begin{defi}
Given an ellipse $E=x_0+\{\lambda e_{short}+\Lambda e_{long}\},$ its eccentricity is defined to be the ratio between its long axis and short axis, namely, $$\eta(E)=\Lambda/\lambda.$$
\end{defi}

\begin{defi}
Given two vectors $v_1$ and $v_2$, we use $\omega(v_1,v_2)$ to denote the angle between them. 

\end{defi}

\begin{defi}
Given a vector $e$ and $0\le\theta\le\pi$, the cone with direction $e$ and opening $\theta$ is defined as $$\Gamma(e,\theta):=\{\lambda v| \lambda>0, \omega(v,e)< \theta\}.$$
\end{defi}

We assume $U_1$ and $U_2$ are two bounded convex domains in $\R^2$ with of area 1. 

We use $\psi:U_1\to\R$ to denote a convex function whose gradient is the optimal transport from $U_1$ to $U_2$, its existence a consequence of \cite{B}.  Moreover, we extend $\psi$ to the entire plane as the following function, still denoted by $\psi$: $$x\mapsto \sup_{y\in U_1} (\psi(y)+\nabla \psi(y)\cdot (x-y)).$$ We use $\varphi$ to denote a convex function whose gradient is the optimal transport from $U_2$ to $U_1$, and extended to $\R^2$ in a similar fashion.  Here $\cdot$ denotes the standard inner product of $\R^2$. 

In the following we often give statements for $\psi$ while omitting analogous ones concerning $\varphi$.

The starting point of the regularity theory is the following observation that the convexity of $U_2$ implies (see \cite{CMap1}):
\begin{prop}
$\psi$  is an Alexandrov solution to $$\det(D^2\psi)=\chi_{U_1} \text{ in $\R^2$}.$$
\end{prop} 

For a systematic introduction to the \MA equation interested readers can consult for example the classic book by Guti\'errez \cite{Giut}, the brief but insightful lectures by Figalli \cite{F} or the book by Le-Mitake-Tran \cite{LMT}. 

Sections are fundamental in the study of the \MA equation. \begin{defi}

The \textit{centred} section of height $h$ of $\psi$ at $x_0$ is $$\Shcpsi(x_0)=\{x\in\R^2| \psi(x)<\psi(x_0)+p\cdot (x-x_0)+h\},$$ where $p\in\R^2$ is chosen so that its centre of mass is $x_0$. 
\end{defi} For the existence of such a vector $p$, see \cite{CMap3}. Next we recall three properties of centered sections which were obtained in \cite{CMap3}.

The first one concerns the engulfing of sections, and it is a consequence of the doubling property of the \MA measure.

\begin{prop}
Given $0<t<\bar{t}<1$, there is $\bar{s}=\bar{s}(t,\bar{t})>0$ such that if $x_1\subset t \, \Shcpsi(x_0)$, then $$S_{sh}^c[\psi](x_1)\subset \bar{t} \, \Shcpsi(x_0) \text{ for all $s<\bar{s}$}.$$
\end{prop} 
Here $\overline{t} \, \Shcpsi(x_0)$ is the dilation with respect to the centre $x_0$ by a factor of $\overline{t}$.

The second property is an area bound for $\Shcpsi[x_0]\cap U_1$.

\begin{prop}\label{p2.4}
There are positive constants $C$ universal and $c>0$ depending only on the diameter of $U_1$ and $U_2$, such that $$ Ch \ge |\Shcpsi[x_0]| \quad \quad \mbox{and} \quad |\Shcpsi[x_0]\cap U_1|\ge c h.$$
\end{prop} 
\begin{proof}
This is proved in Theorem 3.1 in \cite{CMap3}. The bounds hold for all polynomially convex domains with the estimates depending on the dimension $n$, and parameters $\mu$ and $\lambda$, (see Lemma 3.1 in \cite{CMap3}). 
All convex domains in the plane are polynomially convex, and $\mu$ and $\lambda$ only depend on the inner and outer radii of the domains. 

\end{proof} 

A consequence of Proposition \ref{p2.4} is the following result.

\begin{prop}\label{p2.5}
There is a positive constant $\kappa$, depending only on the inner and outer radii of $U_1$ and $U_2$,  such that for each $h>0$ and $x_0\in\overline{\Omega_1}$, we have an ellipse $E_h$ of area equal to $h$ such that up to a translation $$\kappa^{-1}E_h\subset\Shcpsi(x_0) \cap U_1 \quad \mbox {and} \quad  \Shcpsi(x_0) \subset\kappa E_h,$$ 
$$\kappa^{-1}E^\perp_h\subset S_h^c[\varphi](\nabla \psi(x_0)) \cap U_2 \quad \mbox {and} \quad  S_h^c[\varphi](\nabla \psi(x_0))\subset\kappa E_h^\perp.$$ 
\end{prop} 

Such comparison with ellipses allows us to exploit the affine invariance of the problem. To be precise, 
 let $A$ be the affine transformation with $\det A=1$ that maps $E_h$ to a disk. Define the following normalizations  $$\Omega_1= h^{- \frac 12} AU_1 \text{ and } \Omega_2=  h^{- \frac 12}  (A^{-1})^TU_2,$$
$$u(x)=\frac{1}{h}\psi(h^{\frac 12}A^{-1}x) \text{ and } v(y)=\frac{1}{h}\varphi(h^{\frac 12}A^Ty).$$Up to a translation, we might assume $0\in\partial\Omega_1\cap\partial\Omega_2$ and $\nabla u(0)=\nabla v(0)=0.$ Up to a constant, we can also assume $u(0)=v(0)=0.$

Following the definition of $(u,v)$, we know their sections of height $1$ at the origin are comparable to the unit ball up to a factor $\bar \delta$ which depends only on the maximum of the diameters of $U_1$ and $U_2$.  Moreover, the corresponding ellipses for the sections of $u$ and $v$ are dual to each other. 

In particular, all such normalizations of $\psi$ and $\varphi$ belong to the normalized family $\mathcal S(\bar \delta)$ defined below. Even their limits will be contained in the family since we allow unbounded domains in the following definition.

\begin{defi}
For $\bar\delta>0$, we say that $(u,v)\in\mathcal{S}(\bar\delta)$ if $u,v:\R^2\to\R$ are convex functions satisfying the following properties:

\begin{enumerate}
\item{There are (not necessarily bounded) convex sets $\Omega_1$ and $\Omega_2$ such that $$0\in\partial\Omega_1\cap\partial\Omega_2$$ and $\Omega_2=\nabla u(\Omega_1),  \Omega_1=\nabla v(\Omega_2);$ }
\item{$$\det(D^2 u)=\chi_{\Omega_1} \quad \mbox{in $\R^2$, and} \quad \quad u(0)=|\nabla u(0)|=0.$$}
\item{For $h\in (0,1]$, there is a point $x_h$ and an ellipsoid $E_h$ centred at $0$ of volume $h|B_1|$ such that $$x_h+\bar\delta E_h\subset\Omega_1\cap\Shcu(0) \text{ and } \Shcu(0)\subset\bar\delta^{-1}E_h;$$ Moreover, $E_1=B_1$.}
\item{$v$ satisfies similar properties as in 2) and 3) with $\Omega_1$ replaced by $\Omega_2$ and $E_h$ replaced by $E_h^\perp.$ Inside $\Omega_2$, $v$ coincides with the Legendre transform of $u$.}
\end{enumerate}
\end{defi} 

Clearly the class $\mathcal S(\bar \delta)$ remains invariant under the standard affine renormalization. 
Precisely, if $(u,v) \in \mathcal S(\bar \delta)$ then $( u_h, v_h) \in \mathcal S(\bar \delta)$ where
$$u_h(x):= \frac 1 h u(h^{\frac 12} A^{-1}_h  x), \quad v_h(y):= \frac 1h v( h^{\frac 12}A_h^Ty), \quad \quad h \in (0,1],$$
and $A_h$ is an affine transformation (i.e. $\det A_h=1$) which maps $E_h$ into $h^{\frac 12} B_1$.

\begin{defi}
If $(u,v) \in \mathcal S(\bar \delta)$ we denote $$\eta_u(h):=\eta(E_h),$$ as the eccentricity of the section $\Shcu(0)$.
\end{defi}

If $A_h=A_h(u)$ is an affine transformation used in the renormalization above then 
$$\eta_u(h)=\|A_h(u) \|^2.$$ 
Moreover, if $A_t(u_h)$ is an affine transformation which renormalizes the ellipsoid $E_t$ corresponding to $u_h$ then the product
$$A_t(u_h) \, \, A_h(u)$$
is an affine transformation which renormalizes $u_{th}$. In conclusion
\begin{equation}\label{e20}
\eta_u(t h) =\|A_t(u_h) \, \, A_h(u) \|^2 \le \eta_{u}(h) \, \, \eta_{u_h}(t).
\end{equation}

The advantage of working with the family $\mathcal S(\bar \delta)$ is given by the following compactness property.

\begin{prop}[Locally uniform $C^{1,\delta_0}$ estimate and compactness]
There is a dimensional $\delta_0>0$ such that given any compact set $K\subset\R^2$, there is a constant $C$ depending only on $K$ and $\bar\delta>0$ such that $$\|u\|_{C^{1,\delta_0}(K)}\le C \text{ and } \|v\|_{C^{1,\delta_0}(K)}\le C$$ for all $(u,v)\in\mathcal{S}(\bar\delta).$

Moreover, given a sequence $(u_n,v_n)\in\mathcal{S}(\bar\delta)$, there is a subsequence converging locally uniformly to a pair $(u,v)\in\mathcal{S}(\bar\delta).$
\end{prop} 

\begin{proof}
This is essentially Caffarelli's global $C^{1,\delta_0}$-estimate in \cite{CMap2}. Since the domains $\Omega_i$ could be unbounded, we provide a few details.

By Lemma 4 in \cite{CMap2}, (2) in Definition 2.5 implies that there exists $C_0$ depending only on the dimension of the space such that all the sections centred at $x_0,x_1 \in \overline{\Omega_1}$ satisfy the engulfing property
\begin{equation}
x_1\in S^c_t(x_0)  \quad \Longrightarrow \quad 2 \, S_{t}^c[u](x_1) \subset  S_{C_0t}^c[u](x_0)\quad \quad \forall t \ge 0.
\end{equation}

It is not difficult to see that (3) in Definition 2.5  (applied to $x_0=x_1=0$) and $S_1^c[u](0) \sim B_1$ give that for any $R$ large we have
$$B_{4R} \subset S^c_{K_1}[u](0) \subset B_{K_2},$$
 for some appropriate constants $K_i$ depending on $R$ and $\bar\delta$. 
 
 The first inclusion shows that $u \le C(K_1)$ in $B_{2R}$ and therefore $|\nabla u| \le C(R,\bar\delta)$ in $B_R$. 
The second inclusion  shows that for some small $\eps_0>0$, 
$$ S_t^c[u](x_0) \subset B_{C t^ {\eps_0}}(x_0) \quad \mbox{for all $t \le 1$,}$$
which gives a polynomial modulus of strict convexity for $u$ in $\overline{\Omega}_1 \cap B_R$. 

This in turn implies $v$ has bounded $C^{1,\delta_0}$ norm when restricted to the set of points which have supporting planes in $\overline{\Omega_2} \cap B_R$, and the first conclusion of the Proposition follows. See also Remark \ref{r2.1} below. 

Since we already established that $u$, $v$ are uniformly bounded locally, we can extract a convergent subsequence and the fact that properties 1)-4) are preserved under uniform limits on compact sets is standard.  \end{proof}

\begin{rem}\label{r2.1}In \cite{CMap2} it was proved that $v \in C^{1,\delta_0}$ in $\overline{\Omega}_2$, however we show here that this holds in the whole space. 

First we claim that if $z_0$ is a point outside $\overline \Omega_2$, then the supporting plane $l_{z_0}$ to $v$ at $z_0$ coincides with $v$ on some infinite ray that starts at some $y_0 \in \partial \Omega_2$ in the direction of $z_0$. Indeed, since the \MA measure of $v$ vanishes outside $\Omega_2$ we find that all the extremal points of the convex set $\{v=l_{z_0}\}$ belong to $\overline {\Omega}_2$. Since $v$ is strictly convex in $\overline \Omega_2$, the extremal set must have only one point $y_0 \in \partial \Omega_2$. This implies that $\{v=l_{z_0}\}$ is a cone with vertex at $y_0$, and the claim holds.  

Next we show that the supporting planes at $z_0$ and $y_0$ must coincide. Assume for simplicity of notation that $y_0=0$ and $l_{y_0}=0$. If $l_{z_0}\neq 0$ then, we can find a line $s e$, $s \in \mathbb R$, $|e|=1$ passing through $0$, which in the direction $s >0$ points towards the interior of $\Omega_2$, such that the restriction of $v$ to this line is not differentiable at $0$. Now we can use the standard doubling measure argument for sections $S^c_{\mu s}[v](se)$ with $\mu$ sufficiently small, and $ s \to 0^+$, and show that they cannot be balanced with respect to the center $se$ to reach a contradiction.

Finally we consider another point $z_1$ with supporting plane of slope, say $x_1=s e_1$, $s>0$, $|e_1|=1$, and with corresponding ray starting at $y_1 \in \overline {\Omega}_2$.
 Using the polynomial modulus of convexity of $u$ on the segment $[0,x_1]$ we find that
 $$y_1 \cdot x_1 \ge c |x_1|^M \quad \Longrightarrow \quad s \le C (y_1 \cdot e_1)^ {\delta_0}.$$
 Since the rays are infinite the convexity implies that
 $$z_0 \cdot e_1 \le 0, \quad (z_1-y_1) \cdot e_1 \ge 0 \quad \Longrightarrow \quad y_1 \cdot e_1 \le |z_1-z_0|,$$
 hence
 $$|\nabla v(z_1)-\nabla v (z_0)|=s \le C (y_1 \cdot e_1)^{\delta_0} \le C |z_1-z_0|^{\delta_0}.$$ 
\end{rem}
\section{Obliqueness} 
For a pair $(u,v)\in\mathcal{S}(\bar\delta)$, we show that the tangent rays to the domains at corresponding points cannot be perpendicular. 

In a neighborhood $B_r$ of $0$, $\partial\Omega_1$ is the graph of a convex function $\phi$. Up to a rotation, one has 
$$B_r\cap\Omega_1=\{(x_1,x_2)\in\R^2| x_2>\phi(x_1)\}\cap B_r.$$

\textit{The right tangent} to $\Omega_1$ at $0$, to be denoted by $R_{\Omega_1}(0)$, is the unit direction given by the ray starting from $0$ with slope $\lim_{t\to 0^+}\frac{\phi(t)}{t}.$ Symmetrically, \textit{the left tangent} to $\Omega_1$ at $0$, denoted by $L_{\Omega_1}(0)$, is the direction of the ray starting from $0$ with slope $\lim_{t\to 0^-}\frac{\phi(t)}{t}.$ Geometrically, start from any ray pointing outside $\Omega_1$ and rotate it clockwise, the left tangent is the critical ray before  entering $\Omega_1$.  The right tangent is the critical ray if we rotate in the counter-clockwise direction.

Similarly, we can define the left and right tangents to $\partial\Omega_2$.  See Figure 1.

\begin{figure}[h]
\begin{center}
\includegraphics[totalheight=4cm]{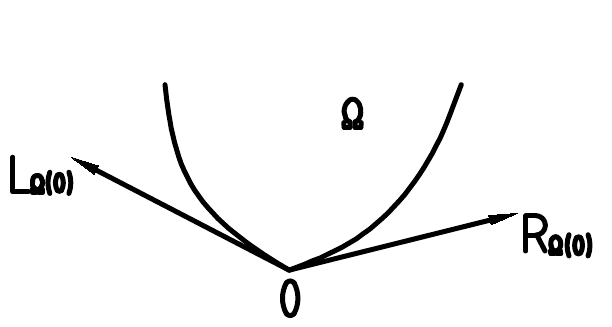}
\caption{Left and right tangents}
\end{center}
\end{figure}

By convexity of $u$ and $v$, one easily sees that
\begin{equation}\label{e30}
\LTO(0)\cdot\LTT(0)\ge 0, \quad \mbox{ and } \quad \RTO(0)\cdot\RTT(0)\ge 0,
\end{equation}
and corresponding tangents are within $\pi/2$ from each other. Obliqueness amounts to ruling out the case where the tangents are exactly perpendicular.

\begin{thm}
$$\LTO(0)\cdot\LTT(0)> 0,$$ and $$\RTO(0)\cdot\RTT(0)> 0.$$
\end{thm} 

\begin{proof}
We only give the argument concerning left tangents. 

Suppose, on the contrary, that \begin{equation}\label{e31}\LTO(0)\cdot\LTT(0)=0.\end{equation}

We denote by $x \in \R^2$ the variables for the potential $u$ and by $y \in \R^2$ the variables for the potential $v$.
After an affine transformation, we may assume that \begin{equation}\label{e32}\omega(\RTT(0),\LTT(0))\ge\frac{2\pi}{3}.\end{equation}
Notice that affine transformations preserve relation \eqref{e31}. 

Finally, we can rotate the coordinate system such that $\LTO(0)$ is the direction of the negative $x_1$-axis, and that $\LTT(0)$ is direction of the positive $y_2$-axis. This means that $\Omega_2 \subset \{y_1 > 0\}$, thus $u$ is nondecreasing in the $x_1$ direction. 

Also, $\partial \Omega_1$ contains the graph of a convex function above the negative $x_1$ axis
$$x_2=\gamma(x_1) \quad \quad \mbox{ with $x_1 \in [-r,0]$, $r>0$ small,}$$ 
which is tangent to the negative $x_1$-axis at the origin. Moreover, by the continuity of the map $\nabla u$, we may assume that $\nabla u$ maps the graph of $\gamma$ onto the graph of a convex function over the $y_2$-axis, which is included in $\partial\Omega_2\cap\{x_2>0\}$.

First we notice that $$ \gamma >0 \quad \mbox{ in $[-r,0)$}.$$
 Otherwise $\partial\Omega_1$ contains a line segment where $u=0$, which is a consequence of the monotonicity of $u$ in the $x_1$ direction together with $u(0)=0$, $\nabla u(0)=0$. This contradicts the strict convexity of $u$ in $\overline{\Omega}_1$. 
 
Meanwhile,  by \eqref{e32}, there is a line segment $se_1$, $ s\in [0, s_0]$ along the positive $y_1$-axis with end point lying inside $\Omega_2$. Denote its image under $\nabla v$ by $\Gamma$. By convexity, and the fact that $v$ is smooth in $\Omega_2$ and is $C^{1,\delta_0}$ in 
$\overline \Omega_2$ we conclude that the curve $\Gamma$ is a graph above $x_2$ axis, smooth except possibly at the origin. Moreover, $\Gamma$ lies in the first quadrant and its endpoint is interior to $\Omega_1$. 

Combining these, we can find some $\delta>0$ such that $\gamma(-r)>2\delta$ and   $\sup_\Gamma x_2>2\delta.$ See Figure 2.

\begin{figure}[h]
\begin{center}
\includegraphics[totalheight=4cm]{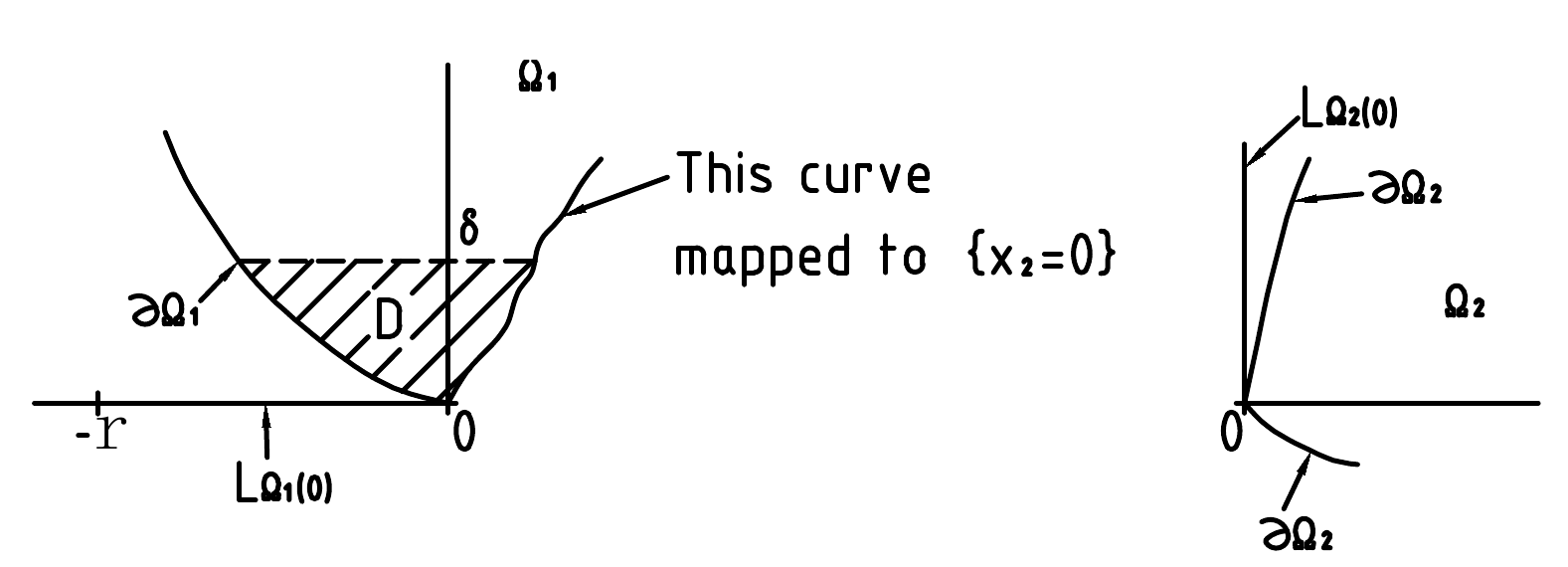}
\caption{}
\end{center}
\end{figure}

\textit{Claim:} There is a constant $C$ such that  
\begin{equation}\label{e33}u(x_1,x_2)\le Cx_2^2 \quad \quad \mbox{if} \quad x_1\le 0, \quad x_2\in [0,\delta].
\end{equation}

Once this claim is established, we have $$S_h[u](0)\supset \left \{(x_1,x_2)|x_1\le 0, x_2\le c h^{1/2} \right\}\cap\Omega_1.$$ Here $\Shu(0)$ is the section of $u$ at $0$ of height $h$ defined as $$\Shu(0)=\{x\in\Omega_1|u(x)<h\}.$$

Since $\partial\Omega_1$ is tangent from the left to the $x_1$-axis, $$\lim_{h\to 0}\frac{|\{(x_1,x_2)|x_1\le 0, x_2\le c h^{1/2}\}\cap\Omega_1|}{h}=\infty.$$ This implies 
$|S_h[u](0)|/h\to\infty$ as $h\to 0$, contradicting the universal volume estimate of sections. This contradiction rules out \eqref{e31}.  

Consequently, to complete the proof, it suffices to prove 
\begin{equation}\label{e36}
u(0,x_2) \le C x_2^2 \quad \quad \mbox{if} \quad x_2 \in [0,\delta],
\end{equation}
which, by the $x_1$-monotonicity of $u$ implies \eqref{e33}. We do this in two steps. In the first step, we establish this inequality under the assumption that $u \in C^2$ up to the boundary on the graph of $\gamma$. In the second step, we remove this restriction by combining step 1 with an approximation argument.

\textit{Step 1:}
In this case, since $\Omega_2\subset\{x_1>0\}$ and $\Omega_2$ is tangent to the positive $x_2$-axis, the image under $\nabla u$ would move to the left when we move along $\partial\Omega_1$ from the left towards $0$, that is, \begin{equation*}
\DDX u(x_1,\gamma(x_1))\ge\DDX u(x_1^*,\gamma(x_1^*)) \text{ if $-r<x_1<x_1^*\le 0$}.
\end{equation*}
Meanwhile, convexity of $u$ implies $\DDX u(x_1^*,\gamma(x_1))\ge\DDX u(x_1,\gamma(x_1)),$ thus $$\DDX u(x_1^*,\gamma(x_1))\ge\DDX u(x_1^*,\gamma(x_1^*)).$$

Now note that $\gamma(x_1)>\gamma(x_1^*)$, and as we let $x_1 \to x_1^*$ we obtain
\begin{equation}\label{e35}
\DDX\DDY u(x_1^*,\gamma(x_1^*))=\DDY\DDX u(x_1^*,\gamma(x_1^*))\ge 0 \text{ for $-r<x_1^*\le 0.$}
\end{equation} 
Note that this is the only step where $C^2$-regularity of $u$ is used, to exchange the order of the two derivatives.

Define $D$ the region above the graphs of $\gamma$ and $\Gamma$, and below the line $x_2= \delta$.

We choose a large constant $C$ such that 
$$\DDY u\le Cx_2 \text{ along $\{x_2=\delta\}$}.$$
Define  $$w:=\DDY u-Cx_2,$$ and let $L$ denote the linearized \MA operator of $u$. Then $$\begin{cases}L \,  w=0 &\text{ in $D$},\\
w\le 0 &\text{ along $\{x_2=\delta\}$}, \\w\le 0 &\text{ along $\Gamma \cap\overline{D}$ },\\ \DDX w\ge 0 &\text{ along $\partial\Omega_1\cap\overline{D}$}.
\end{cases}$$ For the last inequality we used \eqref{e35}. Maximum principle for $L$ gives $w\le 0$ inside $D$ which easily implies the desired inequality \eqref{e36}. This completes the proof of Step 1.  

\textit{Step 2:}
The case for general domains follows from approximation. 

It suffices to establish $w\le 0$ in $D$.  
To see this, take  $$w_\eps:=\DDY u- Cx_2+\eps (x_2^2-1),$$
which is a subsolution to the linearized \MA equation, and notice that $w_\eps \le w \le 0$ on the part of $\partial D$ which lies either on $\{x_2=\delta\}$ or on $\Gamma$. Thus if the maximum of $w_\eps$ in $\overline D$ is positive, then it must occur on the part of $\partial D$ which lies on $\partial \Omega_1$, say at a point $x_0 \in \partial\Omega_1\cap\overline{D}\backslash\{0\}$.
After subtracting from $w_\eps$ the function $\frac 12 \eps[(x-x_0)\cdot e_2]^2$ if necessary, we may assume $w_\eps$ has a strict maximum at $x_0$. 

Now take a tiny neighborhood $U$ of $x_0$ in $\Omega_1$ such that $V=\nabla u(U)$ is convex. Only modifying $U$ and $V$ along $\partial\Omega_1$ and $\partial\Omega_2$ in a neighborhood of $x_0$ and $\nabla u(x_0)$, we can approximate $U$ and $V$ by sequences of sets $U_n$ and $V_n$ which are uniformly convex and smooth in fixed neighborhoods of $x_0$ and $\nabla u(x_0)$. Define $u_n$ and $v_n$ to be the potentials for the optimal transports between $U_n$ and $V_n$. 

Let $w_{\eps,n}:=\DDY u_n-Cx_2+\eps (x_2^2-1)$, and we claim that for large $n$ we have 

a) maximum of $w_{\eps,n}$ in the set $\overline{U}_n$ occurs at a point $x_n$ which converges to $x_0$; 

b) $u_n \in C^{2, \alpha}$ up to the boundary in a fixed neighborhood of $x_0$.

Once we have these, by maximum principle, maximum of $w_{\eps,n}$ in $\overline{U}_n$ lies on $\partial U_n$. Part a) forces the maximum to occur at some  $x_n \in\partial U_n\cap B_r(x_0)$ for  $r$ small. However, along this part of boundary, part b) implies $u_n \in C^2$ and the arguments in Step 1 apply. In particular $\DDX w_{\eps,n}(x_n)\ge 0$ and $w_{\eps, n}$ cannot reach its maximum at $x_n$, a contradiction. 

Next we prove a) and b). By compactness the potentials $u_n$ must converge uniformly (up to constants) to a potential $\bar u$ of the transport map between $U$ and $V$ and $\bar u=u$ in $U$ by uniqueness of optimal transport. Moreover, from our construction, at any point $\bar u$ has a supporting plane with slope in $\overline V$. Since $u \in C^{1,\delta}$, and the supporting planes for $u$ at points near $x_0$ occur in $ \overline V$, we find $\bar u = u$ in some small neighborhood $B_{2 r}(x_0)$. In conclusion $$u_n \to u \quad \mbox{ uniformly in $B_{2r}(x_0)$}.$$
Using that $u \in C^{1,\delta}$ this means that $\|\nabla (u_n - u)\|_{L^\infty} \to 0$ in $B_r(x_0)$ and the part a) of the claim follows.
 
 Now part b) follow from the localized boundary $C^{2,\alpha}$ estimates of \cite{CMap3}. 

\end{proof} 

\section{Growth of eccentricity}
The goal of this section is to show that for $(u,v)\in\mathcal{S}(\bar\delta)$, the eccentricity of the sections grows at most geometrically at a slow rate as one decreases the height. 

To illustrate the idea, assume  $\Shcu(0)$ is highly eccentric. Its normalizing transformation shrinks the direction of $e_{long}(E_h)$  and stretches the orthogonal direction by a large factor.  If $\EL(E_h)$ points well inside $\Omega_1$, then after this normalization $\partial\Omega_1$ becomes almost flat. A Pogorelov-type estimate as in \cite{CMap3} gives the desired estimate.  The same argument works if $\EL(E_h^\perp)$ points well inside $\Omega_2$. 

If neither $E_h$ nor $E_h^\perp$ point well inside the domain, then their long axises must be almost tangent to the domains. Obliqueness forbids the long axises to be tangent to the domains `from the same side', that is, if $\Shcu(0)$ is tangent to $\Omega_1$ `from the left', $\Shcv(0)$ must be tangent to $\Omega_2$ `from the right'. In particular, $\LTO(0)$ is orthogonal to $\RTT(0)$. In Lemma \ref{l42} we show that this special geometry corresponding to critical corner domains again leads to the desired estimate.

We first deal with the case when the long axis points well inside one of the domains. We recall the classical Pogorelov estimate (see Corollary 1.1 in \cite{CMap3}). 

\begin{prop}[Pogorelov estimate] \label{PE}
Assume that $(u,v) \in \mathcal S(\bar \delta)$ and $$\partial \Omega_1 \subset \{x_2=0\} \quad \mbox{ in $S_1^c[u](0)$.}$$
Then
$$ u_{11} \le C \quad \mbox{in $B_c$},$$
for some constants $C$ large, $c$ small, depending only on $\bar \delta$. 
\end{prop}

Next proposition deals with the case when an eccentric section crosses the boundary transversally. The key observation is that the domain is straightening after normalization, and it appears in \cite{CMap3} as well.

\begin{prop}\label{p4.2}Suppose $(u,v)\in\mathcal{S}(\bar\delta)$ and for some angle $\theta>0$ we have
\begin{equation}\label{e41}
\Gamma(\EL(E_h), \theta) \cap S_h^c[u](0) \subset \Omega_1.
\end{equation} 
Given $M$ large, there is a constant $\eta_0$ (large) depending on $M,\theta,\bar\delta$ 
such that $$ \mbox{if $\eta_u(h) \ge \eta_0$ then} \quad \quad \eta_u(\frac 1 M h) \le C_0 \, \eta_u(h),$$
with $C_0=C_0(\bar \delta)$ a constant depending only on $\bar \delta$.
\end{prop}

\begin{proof}

Up to a rotation, $\EL(E_h)$ lies on the $x_2$-axis and that the positive direction is pointing inside $\Omega_1$. Denote this direction by $e_2$. 

Write $\eta=\eta_u(h)$ for simplicity, and let $u_h$ be the rescaling of $u$ which normalizes $E_h$ into $B_1$ obtained as in Section 2 by the affine transformation
$$ A_h(u)=\begin{pmatrix}\sqrt{\eta} & 0\\0&\frac{1}{\sqrt{\eta}}\end{pmatrix}.$$ 
The inclusion \eqref{e41} implies
$$\Gamma(e_2, \tilde \theta) \cap S_1^c[u_h](0) \subset \Omega_1(u_h),$$
with $\tilde{\theta}$ the angle with $$\tan \tilde{\theta}= \eta \, \tan \theta.$$ 

The proof follows by compactness. As $\eta \to \infty$ we have $\tilde \theta \to \pi/2$, and by the compactness of the family $\mathcal S(\delta)$, the rescalings $u_h$ must converge locally uniformly (up to subsequences) to a limiting function $\bar u$ which satisfies the hypothesis of Proposition \ref{PE}. 

Since $\bar u (x_1,0) \le C x_1^2$ we conclude that the ellipsoid $E_{1/M}(u_h)$ intersects the $x_1$ axis on a segment of length $ 2 d M^{-1/2}$ with $d \ge c_0$ for some constant $c_0$ depending only on $\bar \delta$. This means that we can renormalize this ellipsoid of $u_h$ to $B_1$ by using an affine transformation $A_{1/M}(u_h)$ with
$$A_{1/M}(u_h) e_1 = d ^{-1} \, \, e_1 \quad \quad \mbox{with $d \ge c_0$}.$$  
Now we can also use the fact that $\|A_{1/M}(u_h)\| \le C(M)$ to conclude (see \eqref{e20}) that
$$\eta (\frac 1M h) =\|A_{1/M}(u_h) \, A_h(u) \| ^2 \le 2 c_0^{-2}  \, \eta=: C_0 \, \eta,$$
for all large $\eta$.

\end{proof}

The following lemma deals with the critical geometry where the sections are `tangent' to the domains, which can occur near corners. See Figure 3.
\begin{lem}\label{l42}Suppose for some $(u,v)\in\mathcal{S}(\bar\delta)$ we are in the following situation:
\begin{enumerate}
\item{$\Omega_1 \subset\{x_1<0, x_2>0\}$, and $\Omega_2 \subset \{y_1<0, y_2>0\}$;}
\item{$S^c_1[u](0)\cap \{x_1=0, x_2\ge 0\}\subset\partial\Omega_1 \text{ and } S^c_1[v](0)\cap \{y_1\le 0, y_2=0\}\subset\partial\Omega_2$.}
\end{enumerate}

Then there is a constant $C$, depending only on $\bar\delta$, such that $$u(0,x_2)\le C(x_2)_+^2 \text{ in $S_1^c[u](0)$}.$$\end{lem}

\begin{figure}[h]
\begin{center}
\includegraphics[totalheight=4cm]{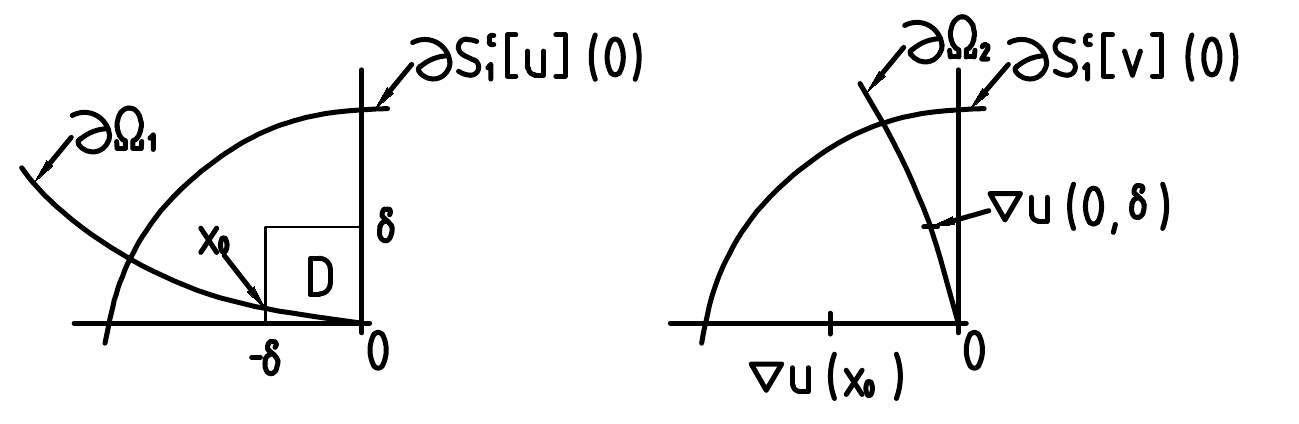}
\caption{}
\end{center}
\end{figure}
\begin{proof}We give the proof under the assumption that $\partial \Omega_1$ and $ \partial \Omega_2$ are smooth away from $0$ in  $S_1^c[u](0)$ and $S_1^c[v](0)$, and their potentials are $C^2$ up to the boundary, except at the origin. The general result follows from an approximation procedure similar to Step 2 as in the proof of Theorem 3.1.

With the geometry described in (1) and (2), the uniform $C^{1,\delta_0}$ estimates for $\mathcal{S}(\bar\delta)$ gives a small $\delta>0$, depending only on $\bar\delta$, such that 
\begin{equation}\label{e43}\nabla u(\{0<x_2<\delta,x_1=0\}) \text{ is a convex graph over the positive $y_2$-axis,}\end{equation}and
\begin{equation}\label{e44}\nabla u(\partial\Omega_1\cap\{-\delta<x_1<0\})\subset\{y_2=0\}.
\end{equation} By choosing $\delta$ smaller if necessary, we also have $$D:=\Omega_1\cap \{x_1>-\delta, x_2<\delta\}\subset S_1^c[u](0).$$

From \eqref{e43} one deduces that when moving upwards along $\{0<x_2<\delta,x_1=0\}$, the image under $\nabla u$ moves in the negative $y_1$-direction, which gives \begin{equation}\label{e45}
\DDY\DDX u\le0 \text{ on $\{0<x_2<\delta,x_1=0\}$}.
\end{equation} 

Meanwhile, if we denote by $x_0=(-\delta,b)\in\partial\Omega_1\cap\{x_1=-\delta\}\cap S_1^c[u](0)$ and $y_0=\nabla u(x_0)$, then $C^{1,\delta_0}$ regularity implies $y_0$ is bounded away from the origin by a small constant depending on $\bar\delta$. With \eqref{e44}, one has that in a neighborhood of size depending only on $\bar\delta$, $\det(D^2v)$ is independent of the $y_1$-variable. 

We can thus apply Pogorelov's estimate to get that $$v_{11}\le C(\bar\delta) \quad \mbox{in some $B_{r(\bar\delta)}(y_0)$.}$$ Duality then gives $u_{22}\le C(\bar\delta) \text{ in $B_{r'(\bar\delta)}(x_0)$}$ for some $r'(\bar\delta)$  depending on $\bar\delta.$ Since $u_2(x_0)=0$, one has \begin{equation*}
u_2 \le C(\bar\delta)x_2 \text{ on $\{x_1=-\delta,b<x_2<b+r'(\bar\delta)\}$}.
\end{equation*}

Since $\nabla u$ is uniformly bounded in $S_1^c[u](0)$ we can choose $C$ large, depending only on $\bar\delta$, such that 
$$u_2\le Cx_2 \quad \mbox{on} \quad \partial D \cap \left(\{x_1=-\delta\} \cup \{x_2= \delta\} \right)$$

Define $$w=u_2-Cx_2,$$ and the inequality above and \eqref{e45} lead to $$\begin{cases}Lw=0 &\text{ in $D$,}\\ w\le 0 &\text{ on $\{x_1=-\delta\}\cup\{x_2=\delta\}\cap\bar{D},$}\\ w\le 0 &\text{ on $\partial\Omega_1\cap\{-\delta<x_1<0\},$}\\ \DDX w\le0 &\text{ on $\partial\Omega_1\cap\{x_1=0\}$,}
\end{cases} $$where $L$ denotes the linearized \MA operator. 

Maximum principle then gives $$w\le 0 \text{ in $D$}.$$Convexity of $u$ leads to the desired estimate for positive $x_2$'s. Since $u$ is nondecreasing in the $x_2$ direction we obtain that $u=0$ on the negative $x_2$ axis, and the proof is finished. 

\end{proof}

With these preparations, we  give the main result of this section:
\begin{prop}\label{p4.3}
Suppose $(u,v)\in\mathcal{S}(\bar\delta)$. 

Given $M$ large, there are constants $\eta_0$ large and $h_0$ small, depending on $M$ and the pair $(u,v),$ such that for $h<h_0$, $$\text{either } 
\eta_u(h)\le \eta_0 \text{ or } \eta_u(\frac{1}{M}h)\le C_1 \eta_u(h),$$ 
with $C_1$ depending only on $\bar\delta$.\end{prop}

\begin{proof}

Suppose the statement is false, and then we find  a sequence $h_n\to 0$ such that \begin{equation}\label{e46}
\eta_u(h_n)\ge n \text{ but }
\eta_u(\frac{1}{M}h_n) > C_1 \eta(h_n),
\end{equation} where $C_1$ is  to be chosen later depending only on $\bar\delta.$

For simplicity of notation we write $\eta_n$ for $\eta_u(h_n)$, and let $e_n$ be the unit direction on the line given by $\EL(E_{h_n})$ which makes an angle at most $\pi/2$ with the direction that bisects the tangent cone of $\Omega_1$ at the origin. We denote $e_n^\perp$ for the perpendicular direction corresponding to $\Omega_2$. 

Without loss of generality, assume $\omega(\LTO(0),e_n)\le\omega(\RTO(0),e_n)$ along this sequence.  We first show that 
\begin{equation}\label{e47}\lim_n \omega(\LTO(0),e_n)=0.\end{equation}

Otherwise we have some $\delta>0$ such that $\omega(\LTO(0),e_n)>\delta$ along the sequence.

There are two possibilities. 
In the first case, $e_n$ is {\it to the right} of $\LTO(0)$ with at least an angle $\delta$ between them. Since $\omega(\LTO(0),e_n)\le\omega(\RTO(0),e_n)$, we obtain $\Gamma(e_n,\delta)\subset \mathcal{K}$, where  $\mathcal{K}$ is the tangent cone of $\Omega_1$ at $0$. Hence Proposition \ref{p4.2} applies with $\theta=\delta/2$ and $h=h_n$ sufficiently small. This contradicts \eqref{e46} if we choose $C_1$ to be larger than the constant $C_0$ in Proposition \ref{p4.2}.

In the second case $e_n$ is {\it to the left of} $\LTO(0)$ with at least an angle $\delta$ between them. Together with $\omega(\LTO(0),e_n)\le\omega(\RTO(0),e_n)$, we contradict the fact that $\bar \delta^{-1} E_{h_n}$ intersects $\Omega_1 \subset \mathcal K$ in a set of area comparable to  the area of $E_{h_n}$, as the eccentricity $\eta_n \to \infty$. In conclusion the claim \eqref{e47} is proved. 

Similar argument applied to $\Omega_2$ gives that either $\lim_n \omega(\LTT(0),e_n^\perp) =0$ or $\lim_n \omega(\RTT(0),e_n^\perp) =0$.  However, obliqueness dictates $$\omega(\LTO(0),\LTT(0))\le\pi/2-\omega_0$$ for some $\omega_0>0$, ruling out the first case.  Consequently, \begin{equation*}\lim_n\omega(\RTT(0),e_n^\perp)=0,\end{equation*} and we conclude \begin{equation}\label{e48}\LTO(0)\perp\RTT(0).\end{equation}

Up to a rotation, we may assume that $\LTO(0)$ lies on the negative $x_1$-axis and $\RTT(0)$ lies on the positive $y_2$-axis. 

Let $A_n$ be the corresponding affine transformation that normalizes $E_{h_n}$ to $B_1$, with eigenvalues $\frac{1}{\sqrt{\eta_n}}$ along $e_n$ and  $\sqrt{\eta_n}$ along $e_n^\perp$ respectively.

Also, denote the normalized solutions by $(u_n,v_n)\in\mathcal{S}(\bar\delta)$ with $\Omega_1^n$ and $\Omega_2^n$ as their corresponding domains. 
Up to a subsequence, they converge to $(u_\infty,v_\infty)\in\mathcal{S}(\bar\delta)$ with domains $\Omega_1^\infty$ and $\Omega_2^\infty.$

Since the angle between $\RTO(0)$ and the positive $x_2$-axis is less than $\pi/2-\omega_0$ for some $\omega_0>0$ and $e_n$ is converging to the negative $x_1$-direction, we find that $R_{\Omega_1^n}$ which points in the direction of $A_n \RTO(0)$ converges $e_2$. Together with the definition of right tangents and $h_n\to 0$, this implies  $$S_1^c[u_\infty](0)\cap\{x_1=0, x_2\ge 0\}\subset\partial\Omega_1^\infty.$$
Similarly, we obtain $$S_1^c[v_\infty](0)\cap\{x_1\le 0, x_2=0\}\subset\partial\Omega_2^\infty.$$

On the other hand, relation \eqref{e48} implies $L_{\Omega_1^n}(0)\perp R_{\Omega_2^n}(0)$ for all $n$. If we denote by $L$ and $R$ the limits of these tangent rays, then $L\perp R$. Also, $\Omega_1^\infty$ is contained between $L$ and the positive $x_2$-axis, $\Omega_2^\infty$ is contained between the negative $y_1$-axis and $R.$  

After an affine transformation $A$ that fixes the $x_2$-axis, we can assume $L$ lies on the negative $x_1$-axis. The dual of $A$ will map $R$ to the positive $y_2$-axis while leaving the $y_1$-axis invariant.  Then $A\Omega_1^\infty$ and $(A^T)^{-1}\Omega_2^\infty$ have the geometry described in Lemma \ref{l42}.  Also, the norms of $A$ and $A^{-1}$ are bounded by a constant depending only on $\bar\delta,$ hence the corresponding potentials belong to a class $\mathcal S(\bar \delta')$ with $\bar \delta'$ depending on $\bar \delta$. By Lemma \ref{l42} we find
$$u_\infty(0,x_2)\le C(\bar \delta)(x_2)_+^2 \text{ in $S_1^c[u_\infty](0)$}.$$ Now we proceed as at the end of the proof of Proposition \ref{p4.2}. The inequality above implies that there is a corresponding affine transformation $A_{1/M}(u_n)$ with
$$ A_{1/M}(u_n) e_2 = d^{-1} e_2, \quad \mbox{with} \quad d \ge c_0(\bar \delta).$$
Since $\|A_{1/M}(u_n)\| \le C(M)$ we find
 $$\eta(\frac 1 M h_n) = \| A_{1/M}(u_n)  \,  A_n \|^2 \le 2 c_0^{-2} \eta_n,$$
for all large $n$, and we reach a contradiction provided that we choose $C_1$ sufficiently large.

\end{proof}

\section{Proof of the main result}In this section we prove the main result. We start with some consequences of Proposition \ref{p4.3}.
\begin{lem}\label{l5.1}
For $(u,v)\in\mathcal{S}(\bar\delta)$ and $\eps>0$, there is a positive constant $h_1$, depending on $\bar\delta,\eps$ as well as the pair $(u,v)$, such that $$\eta_u(h)\le\frac{1}{2}h^{-\eps}$$ whenever $h<h_1.$ 
\end{lem}

\begin{proof}
An iteration of Proposition \ref{p4.3} gives that there exists a constant $K_0$ (depending on $(u,v)$ and $M$) such that
$$\eta_u(M^{-k}h_0)\le K_0 C_1^{k} \text{ for all natural numbers $k$},$$
which implies for general $h < 1$,   $$\eta_u(h)\le K_1 h^{-\log_MC_1}.$$

Since $C_1$ depends only on $\bar\delta$, we can choose $M=M(\bar \delta)$ large enough such that $\log_MC_1<\frac{1}{2}\eps$. 
The desired estimate follows by taking $h_1$ sufficiently small. 

\end{proof} 

Compactness of the family then implies that a uniform version of  Lemma \ref{l5.1} holds at some controlled scale:
\begin{lem}\label{l5.2}
Given $\eps>0$, there is a positive constant $h_0(\eps, \bar \delta)$, depending only on $\bar\delta$ and $\eps$, such that for any $(u,v)\in\mathcal{S}(\bar\delta)$, 
$$\eta_u(h)\le h^{-\eps}$$ for some $h\in [h_0,1/2]$.
\end{lem}
\begin{proof}
Suppose the statement is false, we find sequences $(u_n,v_n)\in\mathcal{S}(\bar\delta)$ and $h_n\to 0$ such that 
\begin{equation}\label{e50}\eta_{u_n}(h)> h^{-\eps} \text{ for all $h\in [h_n,1/2]$.}\end{equation}

Compactness of the family $\mathcal{S}(\bar\delta)$ implies that up to a subsequence, $(u_n,v_n)$ converges locally uniformly to some $(u,v)\in\mathcal{S}(\bar\delta).$ An application of Lemma \ref{l5.1} to $(u,v)$ gives a positive $h_1$ such that $$\eta_u(h_1)\le\frac{1}{2}h_1^{-\eps}.$$ 

Locally uniform convergence of $u_n\to u$ implies  $$\eta_{u_n}(h_1) \le h_1^{-\eps}$$ for large $n$. Note that $h_1\in [h_n,1/2]$ for all large $n$, and we contradict \eqref{e50}.

\end{proof} 

An iteration of Lemma \ref{l5.2} implies the estimate holds true for all $h$.
\begin{thm}\label{t5}
Given $\eps>0$, there is $C=C(\eps,\bar\delta)$ such that for $(u,v)\in\mathcal{S}(\bar\delta)$ $$\eta_u(h)\le Ch^{-\eps}$$ for all $h\le 1.$
\end{thm} 

\begin{proof}
The proof follows from the inequality \eqref{e20}, Lemma \ref{l5.2} and the fact that 
$e_u(h) \le C(\eps, \bar \delta)$ if $h \in [h_0,1]$ with $h_0$ as in Lemma \ref{l5.2}, which is a consequence of the Lipschitz continuity of $u$.

\end{proof}

We finally give the proof of the main result.
\begin{proof}[Proof of Theorem 1.1]
For $y\in U$, let $h_y$ denote the smallest positive number such that $S^c_{h_y}[\psi](y)$ contacts $\partial U_1$, say, at $0$. By the engulfing property, there is a dimensional constant $\theta>1$ such that $$S^c_{\theta h_y}[\psi](0)\supset S^c_{h_y}[\psi](y).$$

Recall that $(\psi,\varphi) \in \mathcal{S}(\bar\delta)$ for some $\bar\delta$ that depends only on largest diameter of the two domains. 

Meanwhile, the bound on $\nabla \psi$, which only depends on the outer radius of $U_2$, implies $$h_y\le Cdist(y,\partial U_1).$$  Consequently, there is some $c>0$ depending only on the outer radius of $U_2$ such that $ \theta h_y<1$ whenever $dis(y,\partial U_1)<c$. 

For all such $y$, Theorem 5.1 and the inclusion above imply that the eccentricity of the section $S^c_{h_y}[\psi](y)$ is bounded by  
 $C(\eps', \bar \delta)h_y^{-\eps'}$ for some $\eps'$ to be chosen later. 
 
 Pogorelov's interior estimate, applied to $S^c_{h_y}[\psi](y)$, gives $$\|D^2\psi(y)\|\le Ch_y^{-\eps'}.$$

The uniform strict convexity of $\varphi$ implies $C |h_y|^{\delta_0}\ge |y|$, thus 
$$\|D^2\psi(y)\|\le C \, \, dist(y,\partial U_1)^{-\eps'/\delta_0}$$ whenever $dist(y,\partial U_1)<c$. We get the desired estimate by choosing $\eps':= \delta_0 \eps $.  

For points with $dist(y, \partial U_1)> c,$ we can apply the interior estimates.

\end{proof} 



\begin{thebibliography}{100}


\bibitem[B]{B} Y. Brenier, {\em D\'ecomposition polaire et r\'earrangement monotone des champs de vecteurs}, C. R. Acad. Sci. Paris S\'er. I. Math. 305 (1987), 805-808. 





\bibitem[C1]{CMA1} L. Caffarelli, {\em A localization property of viscosity solutions to the Monge-Amp\`ere equation and their strict convexity}, Ann. of Math. (2) 131 (1990), no. 1, 129-134.
\bibitem[C2]{CMA2} L. Caffarelli, {\em An interior $W^{2,p}$ estimates for solutions of the Monge-Amp\`ere equation}, Ann. of Math. (2) 131 (1990), no. 1, 135-150.  
\bibitem[C3]{CMA3} L. Caffarelli, {\em Some regularity properties of solutions of Monge-Amp\`ere equation}, Comm. Pure Appl. Math. 44 (1991), no. 8-9, 965-969.
\bibitem[C4]{CMap1} L. Caffarelli, {\em The regularity of mappings with a convex potential}, J. Amer. Math. Soc. 5 (1992), no. 1, 99-104.
\bibitem[C5]{CMap2} L. Caffarelli, {\em Boundary regularity of maps with convex potentials}, Comm. Pure Appl. Math. 45 (1992), no. 9, 1141-1151.
\bibitem[C6]{CMap3} L. Caffarelli, {\em Boundary regularity of maps with convex potentials II}, Ann. of Math. (2) 144 (1996), no. 3, 453-496.


\bibitem[CLW]{CLW} S. Chen, J. Liu, X.J. Wang, {\em Global regularity for the Monge-Amp\`ere equation with natural boundary condition}, eprint arXiv:1802.07518.

\bibitem[DF]{DPF} G. De Philippis, A. Figalli, {\em Partial regularity for optimal transport maps},  Publ. Math. Inst. Hautes \'Etudes Sci. 121 (2015), 81-112. 
\bibitem[F]{F} A. Figalli, {\em The \MA equation and its applications}, Zurich Lectures in Advanced Mathematics. European Mathematical Society, Z\"urich, 2017. 
\bibitem[FK]{FK} A. Figalli, Y.H. Kim, {\em Partial regularity of Brenier solutions of the Monge-Amp\`ere equation}, Discrete Contin. Dyn. Syst. 28 (2010), no. 2, 559-565. 



\bibitem[G]{Giut} C. Guti\'errez, {\em The \MA equation}, Progress in Nonlinear Differential Equations and their Applications, 89. Birkh\"auser/ Springer, Cham, 2016. 
\bibitem[GO]{GO} M. Goldman, F. Otto, {\em A variational proof of partial regularity for optimal transport maps}, eprint arXiv:1704.05339.

\bibitem[LMT]{LMT} N.Q. Le, H. Mitake, H.V. Tran, {\em Dynamical and geometric aspects of Hamilton-Jacobi and linearized \MA equations}, VIASM 2016. Edited by Mitake and Tran. Lecture Notes in Mathematics, 2183. Springer, Cham. 2017. 


\bibitem[S1]{Savin1} O. Savin, {\em Pointwise $C^{2,\alpha}$ estimates at the boundary for Monge-Amp\`ere equation}, J. Amer. Math. Soc. 26 (2013), no. 1,63-99.
\bibitem[S2]{Savin2} O. Savin, {\em Global $W^{2,p}$ estimates for the Monge-Amp\`ere equation}, Proc. Amer. Math. Soc. 141 (2013), no. 10, 3574-3578.



\bibitem[TW]{TW} N. Trudinger, X.J. Wang, {\em Boundary regularity for the Monge-Amp\`ere and affine maximal surface equations}. Ann. of Math. (2) 167 (2008), no. 3, 993-1028.

\bibitem[U]{U} J. Urbas, {\em On the second boundary value problem for equations of \MA type}, J. Reine Angew. Math. 487 (1997), 115-124. 
 
\bibitem[V]{V} C. Villani, {\em Topics in optimal transportation}, Graduate Studies in Mathematics, 58. American Mathematical Society, Providenc, RI, 2003.
\bibitem[W]{Wang} X.J. Wang, {\em Regularity for Monge-Amp\`ere equation near the boundary}. Analysis 16 (1996), no. 1, 101-107.



\end{thebibliography}
\end{document}